\documentclass[reqno]{amsart}
\usepackage{amssymb, amsmath, amsthm, amscd, mathrsfs}

\usepackage{paralist}
\usepackage{cite}
\usepackage{bigints}
\usepackage{dsfont}
\usepackage{empheq}
\usepackage{amssymb}
\usepackage{cases}
\usepackage{enumitem}
\allowdisplaybreaks

\usepackage{caption}
\usepackage{booktabs}

\usepackage{float}
\usepackage[ruled,vlined,linesnumbered,noresetcount]{algorithm2e}

\usepackage{hyperref}

\theoremstyle{plain}
\newtheorem{theorem}{Theorem}[section]

\newtheorem{lemma}[theorem]{Lemma}
\newtheorem{proposition}[theorem]{Proposition}

\theoremstyle{remark}
\newtheorem{remark}[theorem]{Remark}
\numberwithin{equation}{section}

\def \d {\mathrm{d}}

\title[Identification of source terms in wave equation] 
      {Identification of source terms in wave equation with dynamic boundary conditions}

\author{S. E. Chorfi}
\author{G. El Guermai}
\author{L. Maniar}
\author{W. Zouhair}

\address{S. E. Chorfi, G. El Guermai, L. Maniar and W. Zouhair, Cadi Ayyad University, Faculty of Sciences Semlalia, LMDP, UMMISCO (IRD-UPMC), B.P. 2390, Marrakesh, Morocco}

\email{chorphi@gmail.com, ghita.el.guermai@gmail.com, maniar@uca.ma, walid.zouhair.fssm@gmail.com}

\makeatletter 
\@namedef{subjclassname@2020}{%
  \textup{2020} Mathematics Subject Classification}
\makeatother


\begin{document}
\begin{abstract}
This paper studies an inverse hyperbolic problem for the wave equation with dynamic boundary conditions. It consists of determining some forcing terms from the final overdetermination of the displacement. First, the Fréchet differentiability of the Tikhonov functional is studied, and a gradient formula is obtained via the solution of an associated adjoint problem. Then, the Lipschitz continuity of the gradient is proved. Furthermore, the existence and the uniqueness for the minimization problem are discussed. Finally, some numerical experiments for the reconstruction of an internal wave force are implemented via a conjugate gradient algorithm.
\end{abstract}

\maketitle

\section{Introduction}
The wave equation $y_{tt}(t,x) - \Delta y(t,x) = F(t,x)$, posed in a bounded spatial domain $\Omega \subset \mathbb{R}^N$, with boundary $\Gamma :=\partial \Omega$, is a prototype for hyperbolic equations that model the motion in wave phenomena (e.g., acoustic waves, electromagnetic waves, etc). In the one-dimensional framework $N=1$, it models, for instance, the small vibrations of a string subject to some external force $F$. More generally, it reflects the small vibrations of an elastic membrane ($N=2,3$). In this case, $y(t, x)$ is the vertical displacement of the membrane at point $x\in \Omega$ at time $t \in (0,T)$. In addition to the initial conditions, boundary conditions are often taken into account (and sometimes imposed) to characterize the behavior of the solution $y$ on the boundary $\Gamma$. Usually, one deals with static boundary conditions; typically, Dirichlet condition given by the trace $y_{|\Gamma}=g$, Neumann condition given by the normal derivative $\partial_\nu y=g$, and Robin condition $y_{|\Gamma} + \mu \partial_\nu y=g$ that combines both conditions, for a given function $g$ and a constant $\mu \neq 0$. This type of boundary conditions is by now classical in the literature. A less common type is given by dynamic boundary conditions that contain the time derivative of the state on $\Gamma$. In the context of waves, we find, for instance, absorbing boundary conditions \cite{Re'15}:
$$y_t(t,x) + \partial_\nu y(t,x)=G(t,x), \qquad (t,x) \in (0,T)\times \Gamma.$$
We will particularly deal with boundary conditions of type
$$y_{tt}(t,x)+ \partial_\nu y(t,x)=G(t,x), \qquad (t,x) \in (0,T)\times \Gamma;$$
which are also known as kinetic boundary conditions, and are equivalent to Wentzell boundary conditions under some regularity assumption. The physical derivation of such a dynamic boundary condition follows by the principle of stationary action \cite{Go'06}. An other type of dynamic boundary conditions takes the form $$y_{tt}(t,x)- \Delta_\Gamma y+ \partial_\nu y(t,x)=G(t,x), \qquad (t,x) \in (0,T)\times \Gamma,$$
where $\Delta_\Gamma$ denotes the Laplace-Beltrami operator.

In this paper, we investigate an inverse hyperbolic problem consisting of reconstructing some forcing terms in a wave equation with a dynamic boundary condition, from a noisy measured data at final time. Let $T>0$ be a fixed final time and $l>0$ be given. We study an inverse source problem associated with the following hyperbolic problem
\begin{empheq}[left = \empheqlbrace]{alignat=2}
\begin{aligned}
& y_{tt}(t,x) - y_{xx}(t,x)  = F(t,x), &&\qquad \text{in } (0,T)\times(0,l) , \\
& y_{tt}(t,0) - y_{x}(t,0)  = G(t), &&\qquad \text{in } (0,T), \\
& y_{tt}(t,l) + y_{x}(t,l)  = 0, &&\qquad \text{in } (0,T), \\
&(y(0,x), y(0,0),y(0,l)) = (y_{0}(x),a,b), &&\qquad \text{on } (0,l), \\
&(y_{t}(0,x),y_{t}(0,0),y_{t}(0,l)) = (y_{1}(x),c,d),   &&\qquad \text{on } (0,l), \label{eq1to4}
\end{aligned}
\end{empheq}
for some initial displacement $Y_0:=(y_0,a,b)$ and initial velocity $Y_1:=(y_1,c,d)$ belonging to a space that will be given later, and source terms $F\in L^2((0,T)\times (0,l))$ and $G\in L^2(0,T)$. In system \eqref{eq1to4}, $y(t,x)$ represents the displacement of a string of length $l>0$ at time $t$ at the position $x$. The dynamic boundary conditions $\eqref{eq1to4}_2$-$\eqref{eq1to4}_3$ reflect the kinetic energy effects at the ends of the string (the boundary).

There has been considerable interest in inverse problems for hyperbolic equations with static boundary conditions. Stability, reconstruction formula and regularization have been studied by Yamamoto in \cite{Ya'95} for determining spatial component of a source term in a hyperbolic equation using the exact boundary controllability. In \cite{Ha'09'}, Hasanov has proposed a weak solution approach to study the determination of source terms in a linear hyperbolic equation with Neumann boundary conditions from some final overdetermination data. The same machinery has been successfully adapted in \cite{Ha'09} to determine a source term in a vibrating cantilevered beam problem with mixed boundary conditions. In \cite{LHJ'16}, Lesnic et al. have considered an inverse problem for a space-dependent force in the wave equation. In the same scope, we refer to the works of Hussein and Lesnic \cite{Hu'14, Hu'16, Hu'16'}. Recently, an identification problem of a space-time dependent force from many integral observations in a hyperbolic equation with Dirichlet or Robin boundary conditions has been investigated by Alosaimi et al. in \cite{ALD'20}.

In contrast to the large literature for static boundary conditions, there are not sufficient researches on inverse hyperbolic problems incorporating dynamic boundary conditions, in spite of the well-established literature for the direct problems. Some recent works have been lunched for inverse parabolic problems with dynamic boundary conditions \cite{ACM'21, ACM'21', ACMO'20}. As for direct problems, various theoretical approaches have been developed for the analysis of hyperbolic evolution equations with dynamic boundary conditions. In \cite{Mu'06}, Mugnolo has studied the wellposedness of some abstract wave equations with dynamic boundary conditions of acoustic type in the framework of operator matrices. More recently in \cite{Gu'20}, Guidetti has proven some wellposedness results for a general class of mixed hyperbolic systems with dynamic and Wentzell boundary conditions using a semigroup approach. The controllability of a wave equation with oscillatory boundary conditions has been investigated by Gal and Tebou in \cite{CT'17} using Carleman estimates approach.

In this paper, we continue the developments of the weak solution approach for the numerical identification of source forces in the wave equation with dynamic boundary conditions. Solving such a hyperbolic problem with dynamic boundary conditions by using a similar methodology has not been addressed in the literature, as far as we know.

The rest of the paper is organized as follows: in Section \ref{sec2}, we briefly recall some wellposedness and regularity results concerning the system \eqref{eq1to4}. In Section \ref{sec3}, we prove an explicit gradient formula for the Tikhonov functional via the solution of a suitable adjoint problem. Then, we show the Lipschitz continuity of the gradient of the cost functional. In Section \ref{sec4}, we highlight the existence and uniqueness of a quasi-solution to our problem. In Section \ref{sec5}, the theoretical results are validated via a conjugate gradient algorithm designed for the numerical recovery of an unknown internal force. In Section \ref{sec6}, we summarize some conclusions and final remarks.

\section{Wellposedness and regularity of the solution}\label{sec2}
In this section, we briefly present some preliminary results on the wellposedness and the regularity of the solution of system \eqref{eq1to4}. Let us introduce the following real spaces
$$\mathbb{L}^2:=L^2(0,l)\times \mathbb{R}^2 \qquad \text{ and} \qquad \mathbb{L}^2_T:=L^2((0,T)\times(0,l))\times L^2(0,T).$$
$\mathbb{L}^2$ and $\mathbb{L}^2_T$ are Hilbert spaces equipped with the inner products given by
\begin{align*}
\langle (y,a,b),(z,c,d)\rangle_{\mathbb{L}^2} &=\langle y,z\rangle_{L^2(0,l)}+ ac+bd,\\
\langle (y,y_1),(z,z_1)\rangle_{\mathbb{L}^2_T} &=\langle y,z\rangle_{L^2((0,T)\times(0,l))} +\langle y_1,z_1\rangle_{L^2(0,T)},
\end{align*}
respectively. We also introduce the space 
\begin{equation*}
\mathbb{H}^{k}:=\left\{(u, u(0), u(l))\in H^{k}(0, l)\times \mathbb{R}^2\right\} \text { for } k=1,2,
\end{equation*}
equipped with the standard product norm. We define the energy phase space $\mathcal{H}_{\text{en}}:=\mathbb{H}^1\times\mathbb{L}^2$, endowed with the norm
$$\|(y,a,b,z,c,d)\|_{\mathcal{H}_{\text{en}}}= \|(y,a,b)\|_{\mathbb{H}^1}+\|(z,c,d)\|_{\mathbb{L}^2}.$$

We rewrite \eqref{eq1to4} in the following abstract form
\begin{equation*}
\text{(ACP)	} \; \begin{cases}
\hspace{-0.1cm} \partial_t \mathbf{Y}=\mathcal{A} \mathbf{Y}+ \mathcal{F}, \quad 0<t \le T, \nonumber\\
\hspace{-0.1cm} \mathbf{Y}(0)=\mathbf{Y}_0, \nonumber
\end{cases}
\end{equation*}
where $\mathbf{Y}_0:=\left(y_0,a,b, y_1, c,d\right)$, $\mathbf{Y}:=\left(y,y(t,0),y(t,l),y_{t},y_{t}(t,0),y_{t}(t,l)\right)$. The source term is $\mathcal{F}=(0,0,0,F,G,0)$, and the linear operator $\mathcal{A} \colon D(\mathcal{A}) \subset \mathbb{L}^2 \longrightarrow \mathbb{L}^2$ is given by
\begin{equation*}
\mathcal{A}=\begin{pmatrix} O_3 & I_{3} \\ \mathcal{B} & O_3\end{pmatrix}, \quad D(\mathcal{A})=D(\mathcal{B}) \times \mathbb{H}^1, \qquad   \mathcal{B}=\begin{pmatrix} \partial_{xx} & 0 & 0 \\ \partial_{x|x=0} & 0 & 0 \\ -\partial_{x|x=l} & 0 & 0\end{pmatrix} , \quad D(\mathcal{B})=\mathbb{H}^2.
\end{equation*}
The operator $\mathcal{A}$ generates a strongly continuous group $\left(\mathrm{e}^{t\mathcal{A}}\right)_{t\in \mathbb{R}}$ on the phase space $\mathcal{H}_{\text{en}}$, see \cite[Theorem 2.1]{Gu'20} for more details. Consequently, the following result holds.
\newpage

\begin{proposition}
Let $(F,G)\in \mathbb{L}_T^2$.
\begin{itemize}
\item[\text{(a)}] For each $\mathbf{Y}_0\in \mathbb{H}^2\times \mathbb{H}^1$, there exists a unique classical solution to \eqref{eq1to4} such that $\mathbf{Y}\in C^1\left([0,T];\mathbb{H}^1\right)\cap C\left([0,T];\mathbb{H}^2\right)$ and $\mathbf{Y}\in C^2\left([0,T];\mathbb{L}^2\right)$.
\item[\text{(b)}] For each $\mathbf{Y}_0\in \mathcal{H}_{\text{en}}$, there exists a unique mild solution $\mathbf{Y}$ to \eqref{eq1to4} such that $\mathbf{Y}\in C\left([0,T];\mathbb{H}^1\right)\cap C^1\left([0,T];\mathbb{L}^2\right)$.
\end{itemize}
\end{proposition}

\section{Fréchet differentiability and gradient formula of the cost}\label{sec3}
In this section, we consider the following inverse source problem.

\noindent\textbf{Inverse Source Problem (ISP).} Given $(Y_0, Y_1) \in \mathcal{H}_{\text{en}}$, the couple of source terms $(F,G)\in \mathbb{L}^2_T$ in \eqref{eq1to4} is unknown and needs to be recovered from the final displacement at $T$, namely,
$$Y_{T}:=\left(y(T,\cdot),y(T,0),y(T,l) \right)\in \mathbb{L}^2,$$
which is not necessarily smooth due to the numerical noise.

Let $Y(t,\cdot,\mathcal{W})$ be the mild solution of \eqref{eq1to4} corresponding to the source terms $\mathcal{W}=(F,G)\in \mathbb{L}^2_T$. We introduce the input-output operator $\Psi \colon \mathbb{L}^2_T \longrightarrow \mathbb{L}^2$ as follows
$$(\Psi \mathcal{W})(\cdot)=Y_T(\cdot):=Y(T,\cdot,\mathcal{W}) \qquad\text{ on } \;(0,l).$$
Therefore, the ISP with the given output data $Y_T$, can be reformulated as solving the following equation
\begin{equation}\label{2eq2.3}
\Psi \mathcal{W}=Y_T, \quad Y_T \in \mathbb{L}^2,
\end{equation}
which is in turn equivalent to inverting the operator $\Psi$.

In practice, the measurements of $Y_T$ we consider are far to be exact. This shows that we can never fulfill the condition \eqref{2eq2.3} in an exact manner. For this reason, we define a quasi-solution of the considered inverse problem as a solution of the following minimization problem
\begin{align}
\mathcal{J}(\mathcal{W}_*)&= \inf_{\mathcal{W}\in \mathcal{U}_\mathrm{ad}} \mathcal{J}(\mathcal{W}), \label{neq2.5}\\
\mathcal{J}(\mathcal{W})&=\frac{1}{2} \left\|Y(T, \cdot,\mathcal{W})-Y_{T}^\delta \right\|_{\mathbb{L}^2}^2, \qquad \mathcal{W}\in \mathcal{U}_\mathrm{ad}, \label{neq2.6}
\end{align}
where $Y_{T}^\delta=\left(y^\delta(T,\cdot),y^\delta(T,0),y^\delta(T,l) \right)$ is a noisy measured data of $Y_{T}$ such that $\|Y_{T}-Y_{T}^\delta\|\leq \delta$ for some noise level $\delta \geq 0$, and
\begin{align*}
\mathcal{U}_\mathrm{ad} :=\left\{(F,G) \in \mathbb{L}^2_T \colon \begin{array}{ll}
F_* \le F(t,x) \le F^* < +\infty, &\text{ a.e. }(t,x)\in (0,T) \times (0,l) \\
G_* \le G(t) \le G^* < +\infty, &\text{ a.e. } \quad t\in (0,T)
\end{array}\right\}
\end{align*}
is the set of admissible sources. Clearly, $\mathcal{U}_\mathrm{ad}$ is a closed and convex subset of $\mathbb{L}^2_T$.

Due to the ill-posedness of \eqref{2eq2.3}, caused by the compactness of $\Psi$, we usually regularize the problem by considering the following regularized Tikhonov functional
\begin{align*}
\mathcal{J}_\varepsilon(\mathcal{W})&=\frac{1}{2} \left\|Y(T, \cdot,\mathcal{W})-Y_{T}^\delta \right\|_{\mathbb{L}^2}^2 +\frac{\varepsilon}{2} \|\mathcal{W}\|_{\mathbb{L}^2_T}^2, \qquad \mathcal{W}\in \mathcal{U}_\mathrm{ad}, 
\end{align*}
where $\varepsilon >0$ is the regularizing parameter.

Next, let $Y(t,\cdot,\mathcal{W})$ and $Y(t,\cdot,\mathcal{W}+\delta\mathcal{W})$ be the solutions of the direct problem \eqref{eq1to4}, corresponding to the sources $\mathcal{W} = (F,G) $ and $\mathcal{W}+\delta\mathcal{W} = (F+\delta F,G + \delta G) $, respectively. By linearity of the system, $\delta Y := Y(t,\cdot,\mathcal{W}+\delta\mathcal{W}) - Y(t,\cdot,\mathcal{W})$ is the mild solution of the following sensitivity problem
\begin{empheq}[left = \empheqlbrace]{alignat=2}
\begin{aligned}
&\delta y_{tt}(t,x) - \delta y_{xx}(t,x)  = \delta F(t,x), &&\qquad \text{in } (0,T)\times(0,l) , \\
& \delta y_{tt}(t,0) - \delta y_{x}(t,0)  = \delta G(t), &&\qquad \text{in } (0,T), \\
&\delta y_{tt}(t,l) + \delta y_{x}(t,l)  = 0, &&\qquad \text{in } (0,T), \\
&(\delta y(0,x),\delta y(0,0),\delta y(0,l)) = (0,0,0), &&\qquad \text{on } (0,l), \\
&(\delta y_{t}(0,x),\delta y_{t}(0,0),\delta y_{t}(0,l)) =(0,0,0),   &&\qquad \text{on } (0,l). \label{eq3.4}
\end{aligned}
\end{empheq}
Next, we derive an important lemma that will allow us to compute the gradient of $\mathcal{J}$ via the mild solution $\Phi$ of an appropriate adjoint system. This is done following the adjoint methodology in \cite{Du'96, DTB'04}.
\begin{lemma}\label{2prop2.6}
For each $\mathcal{W}\in \mathcal{U}_\mathrm{ad}$, the following integral identity for the cost functional $\mathcal{J}$ holds
\begin{align}\label{equ3.5}
\delta \mathcal{J}(\mathcal{W}) = \int_{0}^{T}\int_{0}^{l}  \varphi\delta F(t,x)  \,\d x \,\d t + \int_{0}^{T} \varphi(t,0)\delta G(t) \,\d t + \frac{1}{2} \|\delta Y (T,\cdot,\mathcal{W} )\|^{2}_{\mathbb{L}^2},
\end{align}
where $\Phi(t,\cdot)=(\varphi(t,\cdot), \varphi(t,0),\varphi(t,l))$ is the mild solution of the following adjoint system
\begin{empheq}[left = \empheqlbrace]{alignat=2}
\begin{aligned}
&\varphi_{tt}(t,x) - \varphi_{xx}(t,x)  = 0, &&\, \text{in } (0,T)\times(0,l) , \\
&\varphi_{tt}(t,0) - \varphi_{x}(t,0)  = 0, &&\, \text{in } (0,T), \\
&\varphi_{tt}(t,l) + \varphi_{x}(t,l)  = 0, &&\, \text{in } (0,T), \\
&(\varphi(T,x), \varphi(T,0),\varphi(T,l)) = (0,0,0) , &&\, \text{on } (0,l), \\
&(\varphi_{t}(T,\cdot),\varphi_{t}(T,0),\varphi_{t}(T,l)) = - \left(Y(T,\cdot,\mathcal{W}) - Y_{T}^\delta \right),   &&\, \text{on } (0,l). \label{eq3.6}
\end{aligned}
\end{empheq}
\end{lemma}
\begin{proof}
Let $\mathcal{W}, \mathcal{W}+ \delta \mathcal{W}\in \mathcal{U}_\mathrm{ad}$. First, we develop the variation
$$\delta \mathcal{J}(\mathcal{W}):=\mathcal{J}(\mathcal{W}+\delta \mathcal{W}) -\mathcal{J}(\mathcal{W}).$$
We have
\begin{align*}
\delta \mathcal{J}(\mathcal{W})&= \frac{1}{2} \left\|Y(T, \cdot, \mathcal{W}+ \delta \mathcal{W})-Y_{T}^\delta \right\|_{\mathbb{L}^2}^2 -\frac{1}{2} \left\|Y(T, \cdot,\mathcal{W})-Y_{T}^\delta \right\|_{\mathbb{L}^2}^2 \nonumber\\
& \hspace{-1.2cm}= \frac{1}{2} \left(\left\|y(T, \cdot, \mathcal{W}+ \delta \mathcal{W})-y_{T}^\delta \right\|_{L^2(0,l)}^2 + \left|y(T,0, \mathcal{W}+ \delta \mathcal{W})-y_{T}^\delta (0) \right|^2 + \left|y(T,l, \mathcal{W}+ \delta \mathcal{W})-y_{T}^\delta (l) \right|^2\right) \nonumber\\
& \hspace{-1cm} - \frac{1}{2}\left(\left\|y(T, \cdot, \mathcal{W})-y_{T}^\delta \right\|_{L^2(0,l)}^2 + \left|y(T,0, \mathcal{W})-y_{T}^\delta (0) \right|^2 + \left|y(T,l, \mathcal{W})-y_{T}^\delta (l)\right|^2\right) \nonumber\\
&\hspace{-1.2cm} = \frac{1}{2} \int_{0}^{l} \left[(y(T, x, \mathcal{W}+ \delta \mathcal{W})-y_{T}^\delta(x))^2  - (y(T, x, \mathcal{W})-y_{T}^\delta(x))^2 \right] \,\d x \\
& \hspace{-1cm}+ \frac{1}{2} \left[(y(T,0, \mathcal{W}+ \delta \mathcal{W})-y_{T}^\delta(0))^2 - (y(T, 0, \mathcal{W})-y_{T}^\delta(0))^2 \right]\\
&\hspace{-1cm} + \frac{1}{2} \left[(y(T,l, \mathcal{W}+ \delta \mathcal{W})-y_{T}^\delta(l))^2 - (y(T, l, \mathcal{W})-y_{T}^\delta(l))^2 \right].
\end{align*}
Since
 $\frac{1}{2} \left[(x-z)^2-(y-z)^2\right]=(y-z)(x-y)+\frac{1}{2}(x-y)^2, \; x,y\in \mathbb{R}$, then
\begin{align}
\delta \mathcal{J}(\mathcal{W}) &= \int_{0}^{l} (y(T, x, \mathcal{W})-y_{T}^\delta(x)) \delta y(T, x, \mathcal{W}) \,\d x +\frac{1}{2} \int_{0}^{l} [\delta y(T, x, \mathcal{W})]^2 \,\d x \label{2eq2.19}\\
& + \left(y(T, 0, \mathcal{W})-y_{T}^\delta(0)\right) \delta y(T, 0, \mathcal{W}) +\frac{1}{2} [\delta y(T, 0, \mathcal{W})]^2 \label{2eq2.18}\\
& + \left(y(T, l, \mathcal{W})-y_{T}^\delta(l)\right) \delta y(T, l, \mathcal{W}) +\frac{1}{2} [\delta y(T, l, \mathcal{W})]^2, \label{2eq2.20}
\end{align}
where
\begin{align*}
\delta y(T, \cdot, \mathcal{W})&= y(T, \cdot, \mathcal{W}+ \delta \mathcal{W})- y(T, \cdot, \mathcal{W}).
\end{align*}
The first integral in the right-hand side of \eqref{2eq2.19} can be expressed using $\delta Y(t,\cdot, \mathcal{W})$, $\Phi(t,\cdot,\mathcal{W})$, the mild solutions of \eqref{eq3.4} and \eqref{eq3.6} respectively. Indeed,
\begin{align}
&\int_{0}^{l} (y(T, x, \mathcal{W})-y_{T}^\delta(x)) \delta y(T, x, \mathcal{W}) \,\d x \nonumber =\int_{0}^{l} -\varphi_{t}(T, x, \mathcal{W}) \delta y(T, x, \mathcal{W}) \,\d x \nonumber\\
&= \bigintss_{0}^{l} \left[\int_0^T -\partial_{t} \left(\varphi_{t}(t, x, \mathcal{W}) \delta y(t, x, \mathcal{W})\right) \,\d t \right]\,\d x \nonumber\\
&= - \int_{0}^{T}\int_{0}^{l} \varphi_{tt} \delta y \,\d x \,\d t - \int_{0}^{T}\int_{0}^{l}  \varphi_{t}\delta y_{t}  \,\d x \,\d t \nonumber\\
&= - \int_{0}^{T}\int_{0}^{l} \varphi_{tt} \delta y \,\d x \,\d t + \int_{0}^{T}\int_{0}^{l}  \varphi\delta y_{tt}  \,\d x \,\d t \nonumber\\
&= - \int_{0}^{T}\int_{0}^{l} \varphi_{xx} \delta y \,\d x \,\d t + \int_{0}^{T}\int_{0}^{l}  \varphi\delta y_{xx}  \,\d x \,\d t+ \int_{0}^{T}\int_{0}^{l}  \varphi\delta F(t,x)  \,\d x \,\d t \nonumber\\
&= \int_{0}^{T}\left[ \varphi \delta y_{x} - \varphi_{x} \delta y \right]^{l}_{0} \,\d t + \int_{0}^{T}\int_{0}^{l}  \varphi\delta F(t,x)  \,\d x \,\d t .\label{eq3.10}
\end{align}
On the other hand, we transform the first term of \eqref{2eq2.18} in the same way as above,
\begin{align}
&(y(T, 0, \mathcal{W})-y_{T}^\delta(0)) \delta y(T, 0, \mathcal{W}) \nonumber = -\varphi_{t}(T, 0, \mathcal{W}) \delta y(T, 0, \mathcal{W}) \,\d x \nonumber\\
&= \int_0^T -\partial_{t} \left(\varphi_{t}(t, 0, \mathcal{W}) \delta y(t, 0, \mathcal{W})\right) \,\d t \nonumber\\
&= - \int_{0}^{T} \varphi_{tt}(t,0) \delta y(t,0) \,\d t - \int_{0}^{T} \varphi_{t}(t,0)\delta y_{t} (t,0) \,\d t \nonumber\\
&= - \int_{0}^{T} \varphi_{tt}(t,0) \delta y(t,0) \,\d t + \int_{0}^{T} \varphi(t,0)\delta y_{tt}(t,0) \,\d t \nonumber\\
&= - \int_{0}^{T} \varphi_{x}(t,0) \delta y(t,0) \,\d t + \int_{0}^{T} \varphi(t,0)\delta y_{x}(t,0) \,\d t + \int_{0}^{T} \varphi(t,0)\delta G(t) \,\d t .\label{eq3.11}
\end{align}
Similarly for \eqref{2eq2.20}, we obtain
\begin{align}
(y(T, l, \mathcal{W})-y_{T}^\delta(l)) \delta y(T, l, \mathcal{W}) \nonumber &= -\varphi_{t}(T, l, \mathcal{W}) \delta y(T, l, \mathcal{W}) \,\d x \nonumber\\
& \hspace{-1cm}=  \int_{0}^{T} \varphi_{x}(t,l) \delta y(t,l) \,\d t - \int_{0}^{T} \varphi(t,l)\delta y_{x}(t,l) \,\d t. \label{eq3.12}
\end{align}
Adding up the three equalities \eqref{eq3.10}, \eqref{eq3.11} and \eqref{eq3.12}, we deduce
\begin{align}
&\int_{0}^{l} (y(T, x, \mathcal{W})-y_{T}^\delta(x)) \delta y(T, x, \mathcal{W}) \,\d x +(y(T, 0, \mathcal{W})-y_{T}^\delta(0)) \delta y(T, 0, \mathcal{W})\nonumber\\
&+ (y(T, l, \mathcal{W})-y_{T}^\delta(l)) \delta y(T, l, \mathcal{W}) = \int_{0}^{T}\int_{0}^{l}  \varphi\delta F(t,x)  \,\d x \,\d t + \int_{0}^{T} \varphi(t,0)\delta G(t) \,\d t. \nonumber
\end{align}
As a result,
\begin{align*}
\delta \mathcal{J}(\mathcal{W}) = \int_{0}^{T}\int_{0}^{l}  \varphi(t,x)\delta F(t,x)  \,\d x \,\d t + \int_{0}^{T} \varphi(t,0)\delta G(t) \,\d t + \frac{1}{2} \|\delta Y (T,\cdot,\mathcal{W} )\|^{2}_{\mathbb{L}^2}.
\end{align*}
This completes the proof.
\end{proof}
Next, we show that the third term on the right-hand side of \eqref{equ3.5} is of order $\mathcal{O}\left(\|\delta \mathcal{W}\|^2_{\mathbb{L}^2_{T}}\right)$.
\begin{lemma}\label{lemma3.2}
Let $\delta y (t,\cdot,\mathcal{W} )$ denote the solution of the problem \eqref{eq3.4} corresponding to the variation $\delta \mathcal{W} \in \mathcal{U}_\mathrm{ad}$. Then the following estimate
holds
\begin{equation}\label{equ3.13}
\|\delta Y (T,\cdot,\mathcal{W} )\|^{2}_{\mathbb{L}^2} \leq 3  T^{3}  \|\delta \mathcal{W}\|^2_{\mathbb{L}^2_{T} }.
\end{equation}
\end{lemma}
\begin{proof}
Multiplying \eqref{eq3.4}$_{1}$ by $\delta y_{t} (t,\cdot )$ and integrating by part over $[0,l]$, we obtain
\begin{equation*}
\frac{1}{2}\frac{\d}{\d t} \int_{0}^{l} (\delta y_{t}(t,x))^2 \, \d x - \left[\delta y_{x}(t,\cdot) \delta y_{t}(t,\cdot) \right]^{l}_{0} + \frac{1}{2} \frac{\d}{\d t}  \int_{0}^{l} (\delta y_{x}(t,x))^2 \, \d x = \int_{0}^{l}\delta F(t,x) \delta y_{t}(t,x) \, \d x .
\end{equation*}
As above, multiplying \eqref{eq3.4}$_{2}$ by $\delta y_{t}(t,0)$, and \eqref{eq3.4}$_{3}$ by $\delta y_{t}(t,l)$, we obtain
\begin{align*}
\frac{1}{2}\frac{\d}{\d t} (\delta y_{t}(t,0))^{2} -  \delta y_{x}(t,0) \delta y_{t}(t,0) &= \delta G(t) \delta y_{t}(t,0),\\
\frac{1}{2}\frac{\d}{\d t} (\delta y_{t}(t,l))^{2} + \delta y_{x}(t,l) \delta y_{t}(t,l) &= 0.
\end{align*}
Adding up the last three formulas, we arrive at
\begin{align*}
  &\frac{1}{2}\frac{\d}{\d t} \int_{0}^{l} (\delta y_{t}(t,x))^2 \, \d x +\frac{1}{2}\frac{\d}{\d t} (\delta y_{t}(t,0))^{2}+\frac{1}{2}\frac{\d}{\d t} (\delta y_{t}(t,l))^{2} + \frac{1}{2}\frac{\d}{\d t}  \int_{0}^{l} (\delta y_{x}(t,x))^2 \, \d x \\
  & \qquad = \int_{0}^{l}\delta F(t,x) \delta y_{t}(t,x) \, \d x + \delta G(t) \delta y_{t}(t,0). 
\end{align*}
Integrating both sides on $[0, t]$, $t \in [0, T ]$, we obtain 
\begin{align}\label{equ3.14}
 \mathcal{L}^{2}(t) = \int_{0}^{t}\int_{0}^{l}\delta F(\tau,x) \delta y_{\tau}(\tau,x) \, \d x\, \d \tau + \int_{0}^{t} \delta G(\tau) \delta y_{\tau}(\tau,0) \, \d \tau,
\end{align}
where
\begin{equation*}
    \mathcal{L}^{2}(t) = \frac{1}{2} \int_{0}^{l} (\delta y_{t}(t,x))^2 \, \d x +\frac{1}{2} (\delta y_{t}(t,0))^{2}+\frac{1}{2} (\delta y_{t}(t,l))^{2}+ \frac{1}{2} \int_{0}^{l} (\delta y_{x}(t,x))^2 \, \d x.
\end{equation*}
Differentiating both sides of \eqref{equ3.14}, we obtain
\begin{align}
 2\mathcal{L}^{\prime}(t)\mathcal{L}(t) = \int_{0}^{l}\delta F(t,x) \delta y_{t}(t,x) \, \d x + \delta G(t) \delta y_{t}(t,0)\qquad \forall t \in [0,T].
\end{align}
By Cauchy-Schwarz inequality, for all $t \in [0,T]$, 
\begin{align}\label{equa3.16}
 2\mathcal{L}^{\prime}(t)\mathcal{L}(t) \leq \left(\int_{0}^{l}(\delta F(t,x))^{2} \, \d x \right)^{\frac{1}{2}} \left(\int_{0}^{l} (\delta y_{t}(t,x))^{2} \, \d x \right)^{\frac{1}{2}} + \left| \delta G(t) \right| \left| \delta y_{t}(t,0)\right|.
\end{align}
Next, we estimate the term $\|\delta y_{t} (t,\cdot,\mathcal{W} )\|^{2}_{\mathbb{L}^2}$ via the energy $\mathcal{L}(t)$. We have
\begin{align}
\int_{0}^{l}  (\delta y_{t}(t,x))^{2} \, \d x &\leq 2 \mathcal{L}(t)^{2} \quad \forall t \in [0,T],  \label{equ3.1.6}\\
\left| \delta y_{t}(t,0)\right|^{2} &\leq 2 \mathcal{L}(t)^{2}   \quad \forall t \in [0,T],\label{equ3.17} \\
\left| \delta y_{t}(t,l)\right|^{2} &\leq 2 \mathcal{L}(t)^{2} \quad \forall t \in [0,T]. \label{equ3.18}
\end{align}
By using \eqref{equa3.16}, \eqref{equ3.1.6} and \eqref{equ3.17}, we obtain
\begin{align*}
 \mathcal{L}^{\prime}(t)\mathcal{L}(t) \leq \frac{1}{\sqrt{2}} \mathcal{L}(t) \left[\left(\int_{0}^{l}(\delta F(t,x))^{2} \, \d x \right)^{\frac{1}{2}} + \left| \delta G(t) \right| \right] \quad \forall t \in [0,T],
\end{align*}
so that
$$\mathcal{L}^{\prime}(t) \leq \frac{1}{\sqrt{2}} \left(\int_{0}^{l}(\delta F(t,x))^{2} \, \d x \right)^{\frac{1}{2}} + \frac{1}{\sqrt{2}} \left| \delta G(t) \right| \qquad  \forall t \in (0,T).$$
Integrating both sides on $[0, t]$ and using the fact that $\mathcal{L} (0) = 0$, we obtain
\begin{align}\label{equa3.19}
 \mathcal{L}(t) \leq \frac{1}{\sqrt{2}} \bigintsss_{0}^{t} \left(\int_{0}^{l}(\delta F(\tau,x))^{2} \, \d x \right)^{\frac{1}{2}} \, \d \tau + \frac{1}{\sqrt{2}}\int_{0}^{t} \left| \delta G(\tau) \right| \, \d \tau \quad\forall t \in [0,T].
\end{align}
By making use of \eqref{equa3.19}, the elementary inequality $(a + b)^2 \leq 2(a^2 + b^2), \; a,b \in \mathbb{R}$, and Cauchy-Schwarz inequality, we find
\begin{align}\label{equ3.21}
2 \mathcal{L}^2(t) &\leq  \left(\bigintsss_{0}^{t} \left(\int_{0}^{l}(\delta F(\tau,x))^{2} \, \d x \right)^{\frac{1}{2}} \, \d \tau + \int_{0}^{t} \left| \delta G(\tau) \right| \, \d \tau\right)^{2} \nonumber \\
 &\leq 2\left(\bigintsss_{0}^{t} \left(\int_{0}^{l}(\delta F(\tau,x))^{2} \, \d x \right)^{\frac{1}{2}} \, \d \tau\right)^{2} + 2\left( \int_{0}^{t} \left| \delta G(\tau) \right| \, \d \tau\right)^{2}\nonumber \\
  &\leq 2 t \int_{0}^{t}\int_{0}^{l}(\delta F(\tau,x))^{2} \, \d x \, \d \tau + 2 t\int_{0}^{t} \left| \delta G(\tau) \right|^{2} \, \d \tau.
\end{align}
By \eqref{equ3.1.6}, \eqref{equ3.17}, \eqref{equ3.18} and  \eqref{equ3.21}, we obtain
\begin{equation}\label{equ3.22}
    \|\delta Y_{t} (t,\cdot )\|^{2}_{\mathbb{L}^2} \leq 6 t\left( \int_{0}^{t}\int_{0}^{l}(\delta F(\tau,x))^{2} \, \d x \, \d \tau + \int_{0}^{t} \left| \delta G(\tau) \right|^{2} \, \d \tau\right).
\end{equation}
Now, let us establish the estimate \eqref{equ3.13}, 
\begin{align*}
 \|\delta Y (T,\cdot,\mathcal{W} )\|^{2}_{\mathbb{L}^2} &= \int_{0}^{l}\left(\delta y(T,x)\right)^{2} \, \d x + \left(\delta y(T,0)\right)^{2}+\left(\delta y(T,l)\right)^{2}\\
  &\hspace{-1.3cm}= \bigintsss_{0}^{l} \left(\int_{0}^{T} \delta y_{t}(t,x)\, \d t\right)^{2} \, \d x + \left(\int_{0}^{T} \delta y_{t}(t,0) \, \d t\right)^{2}+\left(\int_{0}^{T}\delta y_{t}(t,l)\, \d t\right)^{2}\\
  & \hspace{-1.3cm}\leq T\int_{0}^{l}\int_{0}^{T}\left(\delta y_{t}(t,x)\right)^{2} \, \d x\, \d t + T \int_{0}^{T} \left(\delta y_{t}(t,0)\right)^{2}\, \d t+T \int_{0}^{T}\left(\delta y_{t}(t,l)\right)^{2} \, \d t\\
   & \hspace{-1.3cm} \leq T\int_{0}^{T} \|\delta Y_{t} (t,\cdot )\|^{2}_{\mathbb{L}^2} \, \d t\\
    &\hspace{-1.3cm} \leq 6 T \bigintsss_{0}^{T} \left(t \int_{0}^{t}\|\delta F(\tau,\cdot)\|_{L^{2}(0,l)}^{2}  \, \d \tau +  t\int_{0}^{t} \left| \delta G(\tau) \right|^{2} \, \d \tau \right)  \, \d t\\
    & \hspace{-1.3cm} \leq 6 T\int_{0}^{T} t \int_{0}^{t}\|\delta F(\tau,\cdot)\|_{L^{2}(0,l)}^{2}  \, \d \tau \, \d t + 6 T\int_{0}^{T} t\int_{0}^{t} \left| \delta G(\tau) \right|^{2} \, \d \tau   \, \d t\\
    & \hspace{-1.3cm} \leq 3 T \int_{0}^{T} \left(T^2 - t^2\right) \|\delta F(t,\cdot)\|_{L^{2}(0,l)}^{2}  \, \d t + 3 T\int_{0}^{T} \left(T^2 - t^2 \right) \left| \delta G(t) \right|^{2} \, \d t.
\end{align*}
In the third line, we used Cauchy-Schwarz inequality. In the fifth line, the estimate \eqref{equ3.22} is used. Finally, some elementary computations allows us to obtain the desired result. Hence,
\begin{align*}
 \|\delta Y (T,\cdot,\mathcal{W} )\|^{2}_{\mathbb{L}^2}
      &\leq 3 T^3 \int_{0}^{T} \|\delta F(t,\cdot)\|_{L^{2}(0,l)}^{2}  \, \d t + 3 T^3\int_{0}^{T} \left| \delta G(t) \right|^{2} \, \d t,
\end{align*}
which implies the desired estimate \eqref{equ3.13}.
\end{proof}
The above lemma asserts that the term  $\|\delta Y (T,\cdot,\mathcal{W} )\|^{2}_{\mathbb{L}^2}$ is of order $\mathcal{O}\left(\|\delta \mathcal{W}\|^2_{\mathbb{L}^2_{T}}\right)$. Thus, we derive the Fr\'echet gradient of the cost functional $\mathcal{J}$ via the solution of the adjoint system \eqref{eq3.6} as follows
\begin{proposition}\label{prop3.3}
The cost functional $\mathcal{J}$ corresponding to ISP is Fréchet differentiable and its gradient at each $\mathcal{W} \in \mathcal{U}_\mathrm{ad}$ is given by 
$$\mathcal{J}^{\prime}(\mathcal{W}) = \left(\varphi(t,x,\mathcal{W}), \varphi(t,0,\mathcal{W}) \right),$$
where $\left(\varphi(t,x,\mathcal{W}),\varphi(t,0,\mathcal{W}),\varphi(t,l,\mathcal{W})  \right)$ is the solution of the system \eqref{eq3.6}.
\end{proposition}

Next, we prove the Lipschitz continuity of the gradient $\mathcal{J}'$ so that $\mathcal{J}\in C^1(\mathcal{U}_\mathrm{ad})$.
\begin{lemma}\label{2lem2.8}
Then Fréchet gradient $\mathcal{J}^{\prime}$ of the functional $\mathcal{J}$ is Lipschitz continuous. More precisely, the following estimate holds
\begin{equation*}
\left\|\mathcal{J}'(\mathcal{W}+ \delta \mathcal{W})-\mathcal{J}'(\mathcal{W}) \right\|_{\mathbb{L}^2_T} \leq L_{T} \|\delta \mathcal{W}\|_{\mathbb{L}_{T}^2}, 
\end{equation*}
where the Lipschitz constant $L_T>0$ is given by
\begin{equation}\label{lip}
L_{T} = \left[3T^4 \left(l + \frac{l+2}{l} T^2\right) \right]^{\frac{1}{2}} .
\end{equation}
\end{lemma}

\begin{proof}
Denote by $\delta \varphi := \varphi(t,\cdot,\mathcal{W}+\delta \mathcal{W}) - \varphi(t,\cdot,\mathcal{W})$ the strong solution of the following adjoint system
\begin{empheq}[left = \empheqlbrace]{alignat=2}
\begin{aligned}
&\delta\varphi_{tt}(t,x) - \delta\varphi_{xx}(t,x)  = 0, &&\, \text{in } (0,T)\times(0,l) , \\
&\delta\varphi_{tt}(t,0) - \delta\varphi_{x}(t,0)  = 0, &&\, \text{in } (0,T), \\
&\delta\varphi_{tt}(t,l) + \delta\varphi_{x}(t,l)  = 0, &&\, \text{in } (0,T), \\
&(\delta\varphi(T,x),\delta \varphi(T,0),\delta\varphi(T,l)) = (0,0,0) , &&\, \text{on } (0,l), \\
&(\delta\varphi_{t}(T,\cdot),\delta\varphi_{t}(T,0),\delta\varphi_{t}(T,l)) = - \delta Y(T,\cdot,\mathcal{W}),   &&\, \text{on } (0,l). \label{equ3.23}
\end{aligned}
\end{empheq}
Using Proposition \ref{prop3.3}, we have
\begin{align}\label{equ3.25}
\|\mathcal{J}'(\mathcal{W}+ \delta \mathcal{W})-\mathcal{J}'(\mathcal{W})\|^2_{\mathbb{L}^2_T} = \int_{0}^{T}\int_{0}^{l}(\delta\varphi(t,x))^2  \,\d x \,\d t + \int_{0}^{T} (\delta\varphi(t,0))^2 \,\d t.
\end{align}
Next, we estimate the two terms in the right hand side of \eqref{equ3.25}. As in the proof of Lemma \ref{lemma3.2}, multiplying \eqref{equ3.23}$_{1}$, \eqref{equ3.23}$_{2}$ and \eqref{equ3.23}$_{3}$, respectively by $\delta \varphi_{t} (t,x )$, $\delta \varphi_{t} (t,0 )$, and $\delta \varphi_{t} (t,l )$, with a simple calculation, we obtain the following identity
\begin{align*}
\frac{\d}{\d t} \int_{0}^{l} (\delta \varphi_{t}(t,x))^2 \, \d x +\frac{\d}{\d t} (\delta \varphi_{t}(t,0))^{2}+\frac{\d}{\d t} (\delta \varphi_{t}(t,l))^{2}+\frac{\d}{\d t}  \int_{0}^{l} (\delta \varphi_{x}(t,x))^2 \, \d x = 0.
\end{align*}
Integrating over $[t, T ]$ and using the condition \eqref{equ3.23}$_{4}$, we obtain
\begin{align*}
&\int_{0}^{l} (\delta \varphi_{t}(t,x))^2 \, \d x + (\delta \varphi_{t}(t,0))^{2}+ (\delta \varphi_{t}(t,l))^{2}+  \int_{0}^{l} (\delta \varphi_{x}(t,x))^2 \, \d x\\
&= \int_{0}^{l} (\delta \varphi_{t}(T,x))^2 \, \d x + (\delta \varphi_{t}(T,0))^{2}+ (\delta \varphi_{t}(T,l))^{2}.
\end{align*}
By using \eqref{equ3.23}$_{5}$, we obtain
\begin{align*}
 \mathcal{L}^{2}(t) & :=  \|\delta \varphi_{t}(t,\cdot)\|^{2}_{\mathbb{L}^2} +\int_{0}^{l} (\delta \varphi_{x}(t,x))^2 \, \d x\\
 &= \|\delta Y(T,\cdot)\|^{2}_{\mathbb{L}^2}.
\end{align*}
Lemma \ref{lemma3.2} assures that
\begin{equation}\label{equ3.26}
  \mathcal{L}^{2}(t) \leq   3  T^{3}  \|\delta \mathcal{W}\|^2_{\mathbb{L}^2_{T} }.
\end{equation}
Using \eqref{equ3.23}$_{4}$, we deduce that
\begin{align*}
(\delta\varphi(t,x))^{2} &= \left(- \int_{t}^{T} \delta \varphi_{\tau}(\tau,x) \, \d \tau \right)^{2}  \;\qquad\quad \forall t \in [0,T]\\
&\leq (T-t) \int_{t}^{T} \left(\delta \varphi_{\tau}(\tau,x)\right)^{2} \, \d \tau \qquad \forall t \in [0,T].
\end{align*}
Integrating on $[0,T]$, we infer that
\begin{align*}
\int_{0}^{T}(\delta\varphi(t,x))^{2}\, \d t &\leq \int_{0}^{T} (T-t) \int_{t}^{T} \left(\delta \varphi_{\tau}(\tau,x)\right)^{2} \, \d \tau \, \d t\\
&\leq T^{2} \int_{0}^{T}(\delta\varphi_{t}(t,x))^{2}\, \d t.
\end{align*}
Next, we integrate over $[0,l]$, we obtain
\begin{align}\label{equ3.27}
\int_{0}^{l}\int_{0}^{T}(\delta\varphi(t,x))^{2}\, \d t \, \d x &\leq T^{2} \int_{0}^{l} \int_{0}^{T}(\delta\varphi_{t}(t,x))^{2}\, \d t \, \d x.
\end{align}
By using at first \eqref{equ3.27}, then \eqref{equ3.26}, we can estimate the first term in the right-hand side of \eqref{equ3.25} as follows
\begin{align}
\int_{0}^{T}\int_{0}^{l} (\delta\varphi(t,x))^{2}  \,\d x \,\d t &\leq T^{2} \int_{0}^{T} \int_{0}^{l} (\delta\varphi_{t}(t,x))^{2} \, \d x\, \d t\nonumber \\
&\leq T^{2} \int_{0}^{T} \mathcal{L}^{2}(t) \, \d t \nonumber\\
&\leq 3  T^{6}  \|\delta \mathcal{W}\|^2_{\mathbb{L}^2_{T} }. \label{equ3.28}
\end{align}
For the second term in the right-hand side of \eqref{equ3.25}, we make use of the following identity
\begin{align*}
(\delta\varphi(t,0))^{2} &= \left(- \int_{0}^{x} \delta \varphi_{\mu}(t,\mu) \, \d \mu + \delta\varphi(t,x) \right)^{2} \qquad \forall t \in [0,T], \; \forall x \in [0,l].
\end{align*}
Again, by using the inequality $(a+b)^2 \leq 2(a^2 + b^2)$ and then the Cauchy-Schwarz inequality, we obtain
\begin{align*}
(\delta\varphi(t,0))^{2} &\leq 2 \left(\int_{0}^{x} \delta \varphi_{\mu}(t,\mu) \, \d \mu\right)^{2} + 2 \left( \delta\varphi(t,x) \right)^{2}\\
&\leq 2 x \int_{0}^{x} \left(\delta \varphi_{\mu}(t,\mu)\right)^2 \, \d \mu + 2 \left( \delta\varphi(t,x) \right)^{2} \qquad \forall x \in [0,l].
\end{align*}
We integrate over $[0,l]$, with a simple calculation we obtain,
\begin{align*}
 l (\delta\varphi(t,0))^{2}
\leq \int_{0}^{l} \left(l^2 - x^2\right)\left(\delta \varphi_{x}(t,x)\right)^2 \, \d x + 2\int_{0}^{l}  \left( \delta\varphi(t,x) \right)^{2} \, \d x.
\end{align*}
Then,
\begin{align*}
 (\delta\varphi(t,0))^{2}
&\leq l \int_{0}^{l} \left(\delta \varphi_{x}(t,x)\right)^2 \, \d x + \frac{2}{l}\int_{0}^{l}  \left( \delta\varphi(t,x) \right)^{2} \, \d x.
\end{align*}
We integrate on $[0,T]$, so that
\begin{align}\label{equ3.29}
 \int_{0}^{T}(\delta\varphi(t,0))^{2}\, \d t
\leq l\int_{0}^{T} \int_{0}^{l}\left(\delta \varphi_{x}(t,x)\right)^2 \, \d x\, \d t + \frac{2}{l}\int_{0}^{T}\int_{0}^{l}  \left( \delta\varphi(t,x) \right)^{2} \, \d x \, \d t.
\end{align}
By using \eqref{equ3.26}, the first term in right-hand side of \eqref{equ3.29} can be estimated as follows 
\begin{align*}
\int_{0}^{T} \int_{0}^{l}\left(\delta \varphi_{x}(t,x)\right)^2 \, \d x\, \d t &\leq  \int_{0}^{T}\mathcal{L}^2(t) \, \d t\\
&\leq   3  T^{4}  \|\delta \mathcal{W}\|^2_{\mathbb{L}^2_{T}}.
\end{align*}
Thanks to the estimate \eqref{equ3.28}, we obtain
\begin{align}\label{equ3.30}
 \int_{0}^{T}(\delta\varphi(t,0))^{2}\, \d t
&\leq 3 l  T^{4}  \|\delta \mathcal{W}\|^2_{\mathbb{L}^2_{T} } + \frac{6}{l}  T^{6}  \|\delta \mathcal{W}\|^2_{\mathbb{L}^2_{T} }.
\end{align}
Finally, by adding up \eqref{equ3.28} and \eqref{equ3.30}, we obtain
\begin{align*}
\|\mathcal{J}'(\mathcal{W}+ \delta \mathcal{W})-\mathcal{J}'(\mathcal{W})\|^2_{\mathbb{L}^2_T} \leq 3  T^{6}  \|\delta \mathcal{W}\|^2_{\mathbb{L}^2_{T} } + 3 l  T^{4}  \|\delta \mathcal{W}\|^2_{\mathbb{L}^2_{T} } + \frac{6}{l}  T^{6}  \|\delta \mathcal{W}\|^2_{\mathbb{L}^2_{T} }.
\end{align*}
This yields the desired result. 
\end{proof}

Next, we consider the iterative scheme given by
\begin{equation}\label{2eq2.43}
\mathcal{W}_{k+1}= \mathcal{W}_k-\alpha_k \mathcal{J}'(\mathcal{W}_k), \quad k=0,1,2,\dots,
\end{equation}
where $\mathcal{W}_0\in \mathcal{U}_\mathrm{ad}$ is a given initial iteration and $\alpha_k$ is a relaxation parameter defined by the minimization problem
$$h_k(\alpha_k):=\inf_{\alpha\ge 0} h_k(\alpha), \quad h_k(\alpha):=\mathcal{J}\left(\mathcal{W}_k-\alpha \mathcal{J}'(\mathcal{W}_k)\right), \quad k=0,1,2,\dots$$
One can refer to \cite{EHN'00} for more details on such gradient iterations.

The next lemma follows \cite{HR'21}. We denote by $\mathcal{U}_\mathrm{ad}^*$ the set of all quasi-solutions of \eqref{neq2.5}-\eqref{neq2.6}.
\begin{lemma}
Let $(\mathcal{W}_k) \subset \mathcal{U}_\mathrm{ad}$ be the sequence given by \eqref{2eq2.43}. Assuming that the step parameter $\alpha_k =\alpha_{*}$ is constant, it holds that
\begin{enumerate}
\item[(i)] $(\mathcal{J}(\mathcal{W}_k))$ is a monotone decreasing and convergent sequence so that
\begin{equation*}
\lim_{k \to \infty} \|\mathcal{J}'(\mathcal{W}_k)\|_{\mathbb{L}^2_T} =0,
\end{equation*}
and satisfies
\begin{equation*}
\|\mathcal{W}_{k+1} -\mathcal{W}_k\|_{\mathbb{L}^2_T}^2 \le \frac{2}{L_{T}} [\mathcal{J}(\mathcal{W}_k)-\mathcal{J}(\mathcal{W}_{k+1})], \quad k=1,2,3, \ldots,
\end{equation*}
where $L_{T}$ denotes the Lipschitz constant in \eqref{lip}.
\item[(ii)] For each initial iteration $\mathcal{W}_0 \in \mathcal{U}_\mathrm{ad}$, the sequence $(\mathcal{W}_k)$ is weakly convergent in $\mathbb{L}^2_T$ to a quasi-solution $\mathcal{W}_* \in \mathcal{U}_\mathrm{ad}^*$. Moreover, the convergence rate is given by
$$
0 \leq \mathcal{J}(\mathcal{W}_k)-\mathcal{J}_* \leq 2 L_T \frac{\beta^2}{k}, \quad k=1,2,3, \ldots,
$$
where $\mathcal{J}_* :=\lim\limits_{k \to \infty} \mathcal{J}\left(\mathcal{W}_k\right)$ and $\beta:=\sup \left\{\left\|\mathcal{W}_k-\mathcal{W}_*\right\|_{\mathbb{L}^2_T} \colon \mathcal{W}_k\in \mathcal{U}_\mathrm{ad}, \mathcal{W}_* \in \mathcal{U}_\mathrm{ad}^* \right\}.$
\end{enumerate}
\end{lemma}

\section{Existence and uniqueness of a quasi-solution}\label{sec4}
In the following lemma, we show the convexity of the functional $\mathcal{J}$ via the monotonicity of $\mathcal{J}'$. It will allow us to exhibit a sufficient condition for uniqueness of a quasi-solution to ISP.
\begin{lemma}\label{monlem}
The functional $\mathcal{J}\in C^1(\mathcal{U}_\mathrm{ad})$ satisfies following formula
\begin{equation*}
\langle \mathcal{J}'(\mathcal{F} + \delta \mathcal{F})- \mathcal{J}'(\mathcal{W}),\delta \mathcal{W}\rangle_{\mathbb{L}^2_T}= \|\delta Y(T,\cdot, \mathcal{W})\|^2_{\mathbb{L}^2}, \quad \forall \mathcal{W}, \delta \mathcal{W} \in \mathcal{U}_\mathrm{ad},
\end{equation*}
where $\delta Y(T,\cdot, \mathcal{W})$ denotes the solution of \eqref{eq3.4}.
\end{lemma}

\begin{proof}
Let $(\delta \varphi, \delta \varphi(\cdot,0),\delta \varphi(\cdot,l))$ be the solution of system \eqref{equ3.23}. Proposition \ref{prop3.3} implies the following formula,
\begin{align*}
&\langle \mathcal{J}'(\mathcal{W} + \delta \mathcal{W})- \mathcal{J}'(\mathcal{W}),\delta \mathcal{W}\rangle_{\mathbb{L}^2_T} = \int_{0}^{T}\int_{0}^{l} \delta F(t,x) \delta \varphi(t,x) \,\d x\, \d t + \int_{0}^{T} \delta G(t) \delta \varphi(t,0)\, \d t \notag\\
&= \int_{0}^{T}\int_{0}^{l} \delta y_{tt}(t,x) \delta \varphi(t,x) \,\d x\, \d t - \int_{0}^{T}\int_{0}^{l} \delta y_{xx}(t,x) \delta \varphi(t,x) \,\d x\, \d t + \int_{0}^{T} \delta G(t) \delta \varphi(t,0)\, \d t .\notag
\end{align*}
We integrate by parts the first term in the above formula to obtain
\begin{align}
 &\int_{0}^{T}\int_{0}^{l} \delta y_{tt}(t,x) \delta \varphi(t,x) \,\d x\, \d t \nonumber\\
 &= \int_{0}^{l} \left[\delta y_{t}(t,x) \delta \varphi(t,x) \right]_{0}^{T}\,\d x -  \int_{0}^{l}\int_{0}^{T} \delta y_{t}(t,x) \delta \varphi_{t}(t,x) \,\d t\, \d x  \nonumber\\
 &= - \int_{0}^{l} \left[\delta y(t,x) \delta \varphi_{t}(t,x) \right]_{0}^{T}\,\d x + \int_{0}^{l}\int_{0}^{T} \delta y(t,x) \delta \varphi_{tt}(t,x) \,\d t\, \d x  \nonumber\\
  &= \int_{0}^{l} \left( \delta y(T,x) \right)^{2}\,\d x + \int_{0}^{l}\int_{0}^{T} \delta y(t,x) \delta \varphi_{tt}(t,x) \,\d t\, \d x .  \label{eq4.2}
\end{align}
The same technique applied to the second term yields
\begin{align}
 &\int_{0}^{T}\int_{0}^{l} \delta y_{xx}(t,x) \delta \varphi(t,x) \,\d x\, \d t = \int_{0}^{T} \left[\delta y_{x}(t,x) \delta \varphi(t,x) \right]_{0}^{l}\,\d t -  \int_{0}^{T}\int_{0}^{l} \delta y_{x}(t,x) \delta \varphi_{x}(t,x)\, \d x\,\d t \nonumber\\
 &= \int_{0}^{T} \left[\delta y_{x}(t,x) \delta \varphi(t,x) \right]_{0}^{l}\,\d t - \int_{0}^{T} \left[\delta y(t,x) \delta \varphi_{x}(t,x) \right]_{0}^{l}\,\d t +  \int_{0}^{T}\int_{0}^{l} \delta y(t,x) \delta \varphi_{xx}(t,x)\, \d x\,\d t \nonumber\\
 &= \int_{0}^{T} \delta y_{x}(t,l) \delta \varphi(t,l)\,\d t - \int_{0}^{T} \delta y_{x}(t,0) \delta \varphi(t,0)\,\d t -\int_{0}^{T} \delta y(t,l) \delta \varphi_{x}(t,l)\,\d t \nonumber \\
 & \quad + \int_{0}^{T} \delta y(t,0) \delta \varphi_{x}(t,0)\,\d t + \int_{0}^{T}\int_{0}^{l} \delta y(t,x) \delta \varphi_{xx}(t,x)\, \d x\,\d t .\label{eq4.3}
\end{align}
Let us calculate the first term in the right-hand side of \eqref{eq4.3},
\begin{align}
\int_{0}^{T} \delta y_{x}(t,l) \delta \varphi(t,l)\,\d t &= - \int_{0}^{T} \delta y_{tt}(t,l) \delta \varphi(t,l)\,\d t \nonumber\\
&= - \left[\delta y_{t}(t,l)\delta \varphi(t,l)  \right]_{0}^{T} + \int_{0}^{T} \delta y_{t}(t,l) \delta \varphi_{t}(t,l)\,\d t \nonumber\\
&= \left[\delta y(t,l)\delta \varphi_{t}(t,l)  \right]_{0}^{T} - \int_{0}^{T} \delta y(t,l) \delta \varphi_{tt}(t,l)\,\d t \nonumber\\
&= -\left(\delta y(T,l)  \right)^{2} + \int_{0}^{T} \delta y(t,l) \delta \varphi_{x}(t,l)\,\d t . \label{eq4.4}
\end{align}
Similarly for the second term in the right-hand side of \eqref{eq4.3}, we have
\begin{align}
&\int_{0}^{T} \delta y_{x}(t,0) \delta \varphi(t,0)\,\d t = \int_{0}^{T} \delta G(t) \delta \varphi(t,0)\,\d t  - \int_{0}^{T} \delta y_{tt}(t,0) \delta \varphi(t,0)\,\d t \nonumber\\
&=\int_{0}^{T} \delta G(t) \delta \varphi(t,0)\,\d t - \left[\delta y(t,0)\delta \varphi_{t}(t,0)  \right]_{0}^{T} - \int_{0}^{T} \delta y_{t}(t,0) \delta \varphi_{tt}(t,0)\,\d t \nonumber\\
&=\int_{0}^{T} \delta G(t) \delta \varphi(t,0)\,\d t -  \left(\delta y(T,0) \right)^{2} - \int_{0}^{T} \delta y(t,0) \delta \varphi_{tt}(t,0)\,\d t \nonumber\\
&=\int_{0}^{T} \delta G(t) \delta \varphi(t,0)\,\d t -\left(\delta y(T,0)  \right)^{2} - \int_{0}^{T} \delta y(t,0) \delta \varphi_{x}(t,0)\,\d t .\label{eq4.5}
\end{align}
By making use of \eqref{eq4.3}-\eqref{eq4.5}, we obtain
\begin{align}
\int_{0}^{T}\int_{0}^{l} \delta y_{xx}(t,x) \delta \varphi(t,x) \,\d x\, \d t &=  \int_{0}^{T} \delta G(t) \delta \varphi(t,0)\,\d t -\left(\delta y(T,0)  \right)^{2} -\left(\delta y(T,l)  \right)^{2}\nonumber \\
& \quad +\int_{0}^{T}\int_{0}^{l} \delta y(t,x) \delta \varphi_{xx}(t,x)\, \d x\,\d t .\label{eq4.6}
\end{align}
Finally, thanks to \eqref{eq4.2} and \eqref{eq4.6}, we deduce
\begin{equation*}
\langle \mathcal{J}'(\mathcal{F} + \delta \mathcal{F})- \mathcal{J}'(\mathcal{W}),\delta \mathcal{W}\rangle_{\mathbb{L}^2_T}= \int_{0}^{l} \left( \delta y(T,x) \right)^{2}\,\d x + \left(\delta y(T,l)  \right)^{2} +\left(\delta y(T,l)  \right)^{2}.
\end{equation*}
This completes the proof.
\end{proof}

Since the the cost functional $\mathcal{J}$ is continuous and convex on $\mathcal{U}_\mathrm{ad}$, the problem \eqref{neq2.5}-\eqref{neq2.6} has at least one solution on $\mathcal{U}_\mathrm{ad}$, see \cite[Theorem 25.C]{Ze'90}. 
On the other hand, the strict monotonicity of $\mathcal{J}'$ implies the strict convexity of $\mathcal{J}$. Then, if in addition the following condition holds
\begin{equation}\label{eq4.7}
  \|\delta Y(T,\cdot, \mathcal{W})\|^2_{\mathbb{L}^2} > 0 \qquad \forall \mathcal{W} \in \mathcal{V}  
\end{equation}
for a closed convex subset $\mathcal{V} \subset \mathcal{U}_\mathrm{ad}$, then the problem \eqref{neq2.5}-\eqref{neq2.6} admits at most one solution in $\mathcal{V}$. Note that the non-uniqueness of a quasi-solution occurs in the general case of time-space dependent sources.

\section{Numerical results for one spatial wave force}\label{sec5}
In the following, we will present some numerical tests obtained using the quasi-solution approach combined with a Conjugate Gradient (CG) algorithm. We seek to reconstruct the spatial force $f(x)$ in the following wave equation with dynamic boundary conditions
\begin{empheq}[left = \empheqlbrace]{alignat=2}
\begin{aligned}
&y_{tt}(t, x)-y_{x x}(t, x)=f(x)r(t,x), &&\qquad(t,x)\in (0,T)\times (0,l)  , \\
&y_{tt}(t, 0) - y_{x}(t, 0)=0, &&\qquad t\in (0,T), \\
&y_{tt}(t, l) + y_{x}(t, l)=0, &&\qquad t\in (0,T), \\
&(y(0,x), y(0,0),y(0,l)) = (y_{0}(x),a,b), &&\qquad \text{on } (0,l), \\
&(y_{t}(0,x),y_{t}(0,0),y_{t}(0,l)) = (y_{1}(x),c,d),   &&\qquad \text{on } (0,l),
\label{1deq1to4}
\end{aligned}
\end{empheq}
where the space-time dependent component $r \in C^1\left([0,T]; C\left([0,l]\right)\right)$ is a known function. Note that such a form of source terms is needed to fulfill the uniqueness of the solution, see e.g. \cite{CD'70}. Furthermore, it covers many practical applications from control theory, among other fields.

We design an iterative algorithm based on the theoretical study carried out in the previous sections. Let $Y(t,x,f):=\left(y(t,x),y(t,0),y(t,l)\right)$ denote the solution of \eqref{1deq1to4}. The input-output operator $\Psi \colon L^2(0,l) \longrightarrow L^2(0,l)\times \mathbb{R}^2$ is defined as
$$(\Psi f)(x):=Y(T,x,f)=\left(y(T,x),y(T,0),y(T,l)\right), \qquad x\in (0,l),$$
and the corresponding Tikhonov functional is given by
\begin{align*}
J_\varepsilon(f)&=\frac{1}{2} \left\|Y(T, \cdot,f)-Y_{T}^\delta \right\|_{L^2(0,l)\times \mathbb{R}^2}^2 +\frac{\varepsilon}{2} \|f\|_{L^2(0,l)}^2, \qquad f\in L^2(0,l), \\
& \hspace{-0.3cm}= \frac{1}{2} \left(\left\|y(T, \cdot)-y_{T}^\delta\right\|_{L^2(0,l)}^2 + \left|y(T,0)-y_T^{0,\delta}\right|^2 + \left|y(T,l)-y_T^{l,\delta}\right|^2 + \varepsilon \|f\|_{L^2(0,l)}^2\right), \notag
\end{align*}
where $Y_{T}^\delta:=\left(y_{T}^\delta, y_T^{0,\delta}, y_T^{l,\delta}\right) \in L^2(0,l)\times \mathbb{R}^2$.
The corresponding adjoint system is given by
\begin{empheq}[left = \empheqlbrace]{alignat=2}
\begin{aligned}
& \varphi_{tt}(t, x)-\varphi_{x x}(t, x)=0, &&\hspace{-1cm} (t,x)\in (0,T)\times (0,l)  , \\
& \varphi_{tt}(t, 0) - \varphi_{x}(t, 0)=0, && t\in (0,T), \\
& \varphi_{tt}(t, l) + \varphi_{x}(t, l)=0, && t\in (0,T), \\
&(\varphi(T,x), \varphi(T,0),\varphi(T,l)) = (0,0,0) , &&\, \text{on } (0,l), \\
&(\varphi_{t}(T,x),\varphi_{t}(T,0),\varphi_{t}(T,l)) = - (Y(T,x,\mathcal{W}) - Y_{T}^\delta ),   &&\, \text{on } (0,l).
\label{1daeq1to4}
\end{aligned}
\end{empheq}
A simple calculation shows that the gradient of $J_\varepsilon$ is given by
\begin{equation}\label{1dj'}
J_\varepsilon'(f)(x)=\int_0^T \varphi(t,x,f) r(t,x) \,\d t + \varepsilon f(x), \quad f\in L^2(0,l),\;  x\in (0,l).
\end{equation}
This formula for the gradient of $J_\varepsilon$ gives the possibility to apply various CG algorithms corresponding to different coefficients. 

Next, we define the convergence error and the accuracy error respectively by
\begin{equation}\label{err}
\begin{aligned}
e(k,f_k)&:=\|\Psi f_k -Y_T\|_{L^2(0,1)\times \mathbb{R}^2}^2\\
E(k,f_k)&:=\|f-f_k\|_{L^2(0,1)}.
\end{aligned}
\end{equation}

We perturb the exact data by different levels of noise and compare the exact source term to the recovered one. The noisy measured data is generated as follows
$$Y_T^\delta(x) =Y_T(x) + p \times\|Y_T\|_{L^2(0,l)\times \mathbb{R}^2} \times \mathrm{Random},$$
where $p$ designates the noise level, and the function `Random' produces random real numbers.

In all numerical experiments, we take for simplicity the following values
$$T=2, \quad l=1, \quad r=1, \quad y_0=y_1=0, \quad a=b=c=d=0.$$

In what follows, we apply the following CG algorithm:
\begin{algorithm}
 \caption{CG algorithm}\label{alg1}
 Set $k=0$ and choose an initial source $f_0$\;
 Solve the direct problem \eqref{1deq1to4} to obtain $Y(t,x,f_0)$\;
 Knowing the computed $Y(T,x,f_0)$ and the measured $Y_T^\delta$, solve the adjoint problem \eqref{1daeq1to4} to obtain $\varphi(t,x,f_0)$\;
 Compute the gradient $p_0=J_\varepsilon'(f_0)$ using \eqref{1dj'}\;
 Solve the direct problem \eqref{1deq1to4} with source $p_k$ to obtain the solution $\Psi p_k$\;
 Compute the relaxation parameter $\displaystyle \alpha_k =\frac{\|J_\varepsilon'(f_k)\|_{L^2(0,l)}^2}{\|\Psi p_k\|_{L^2(0,l)\times \mathbb{R}^2}^2 + \varepsilon \|p_k\|_{L^2(0,l)}^2}$ \;
 Find the next iteration $f_{k+1}=f_k- \alpha_k p_k$\;
 Stop the iteration process if the stopping criterion $J_\varepsilon(f_{k+1}) <e_J$ holds. Otherwise, set $k:=k+1$ and compute
 $$\gamma_k=\frac{\|J_\varepsilon'(f_k)\|_{L^2(0,l)}^2}{\|J_\varepsilon'(f_{k-1})\|_{L^2(0,l)}^2} \qquad \text{and} \qquad p_k=J_\varepsilon'(f_k)+\gamma_k p_{k-1},$$
 and go to Step 6.
\end{algorithm}

Next, the initial iterations are chosen as $f_0=0$ and the regularization parameter and the stopping parameter as $\varepsilon=e_J=10^{-8}$. 

\subsection*{Example 1}
The exact source term to be reconstructed is $$f(x)=\frac{1}{2}\left(\sin(\pi x)+\sqrt{x}\right), \; x\in (0,1).$$

\begin{figure}[ht] 
\centering
\includegraphics[scale=0.5]{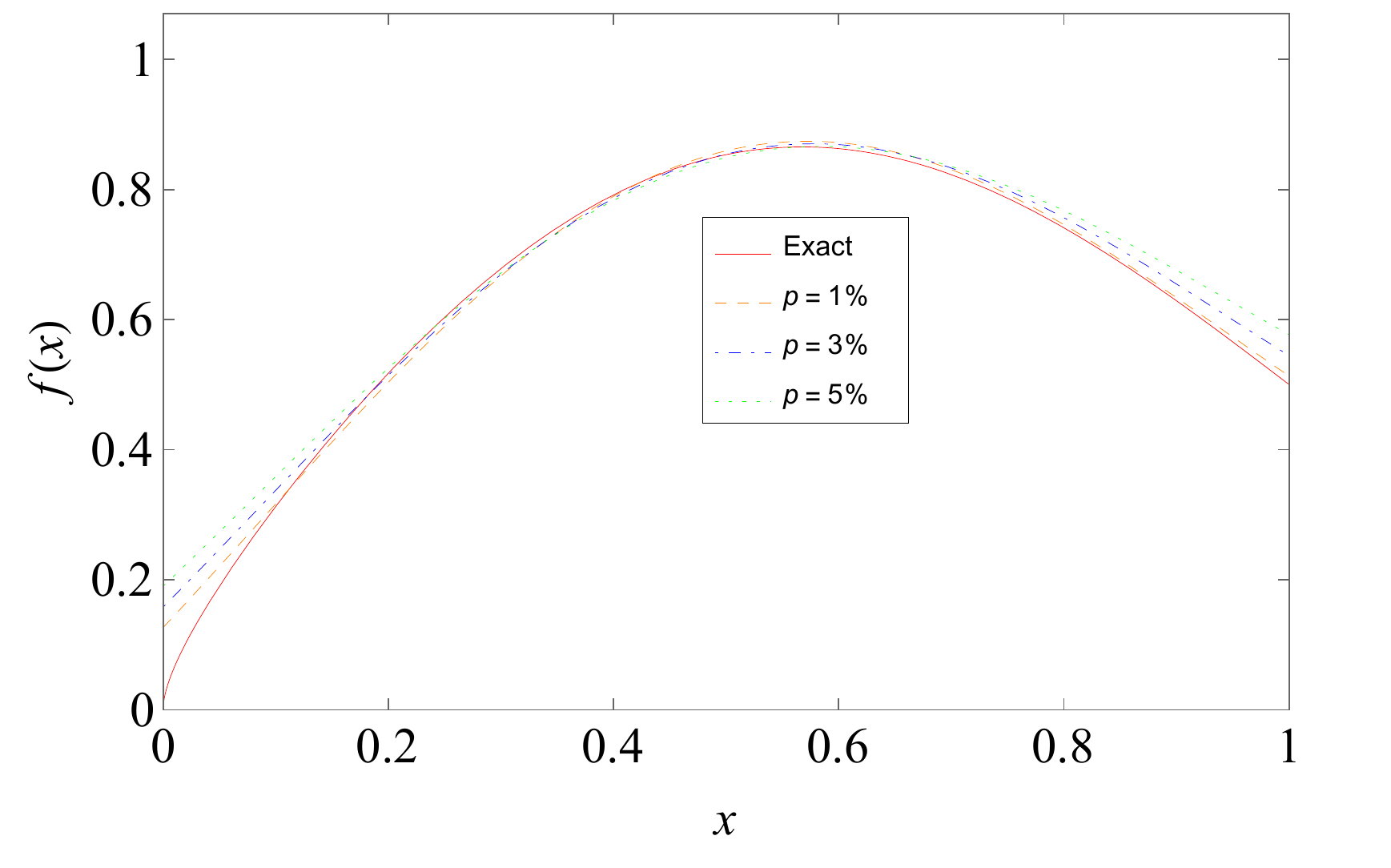}
\caption{Exact and recovered $f(x)$ by using Algorithm \ref{alg1} for $p\in \{1\%,3\%,5\%\}$, respectively.} \label{fig1}
\end{figure}

The algorithm stops at iterations $k\in \{6,8,9\}$, for $p\in \{1\%,3\%,5\%\}$, respectively.

\begin{table}[ht] 
\caption{Errors depending on the iteration number $k$ for noise free data ($p=0\%$).}
\centering
$\begin{array}{cccccc}
\hline
 k & 1 & 2 & 3 & 4 & 5 \\
\hline
 e\left(k,f_k\right) & 7.547\times 10^{-1} & 4.217\times 10^{-2} & 3.447\times 10^{-3} & 3.445\times 10^{-3} & 2.108\times 10^{-3} \\
\hline
 \text{E}\left(k,f_k\right) & 2.015\times 10^{-1} & 1.747\times 10^{-1} & 1.746\times 10^{-1} & 6.744\times 10^{-2} & 1.526\times 10^{-2} \\
\hline
\end{array}$ \label{tab1}
\end{table}

\subsection*{Example 2}
The exact source term is $$f(x)=2 \pi x^2(1-x), \; x\in (0,1).$$

\begin{figure}[ht] 
\centering
\includegraphics[scale=0.5]{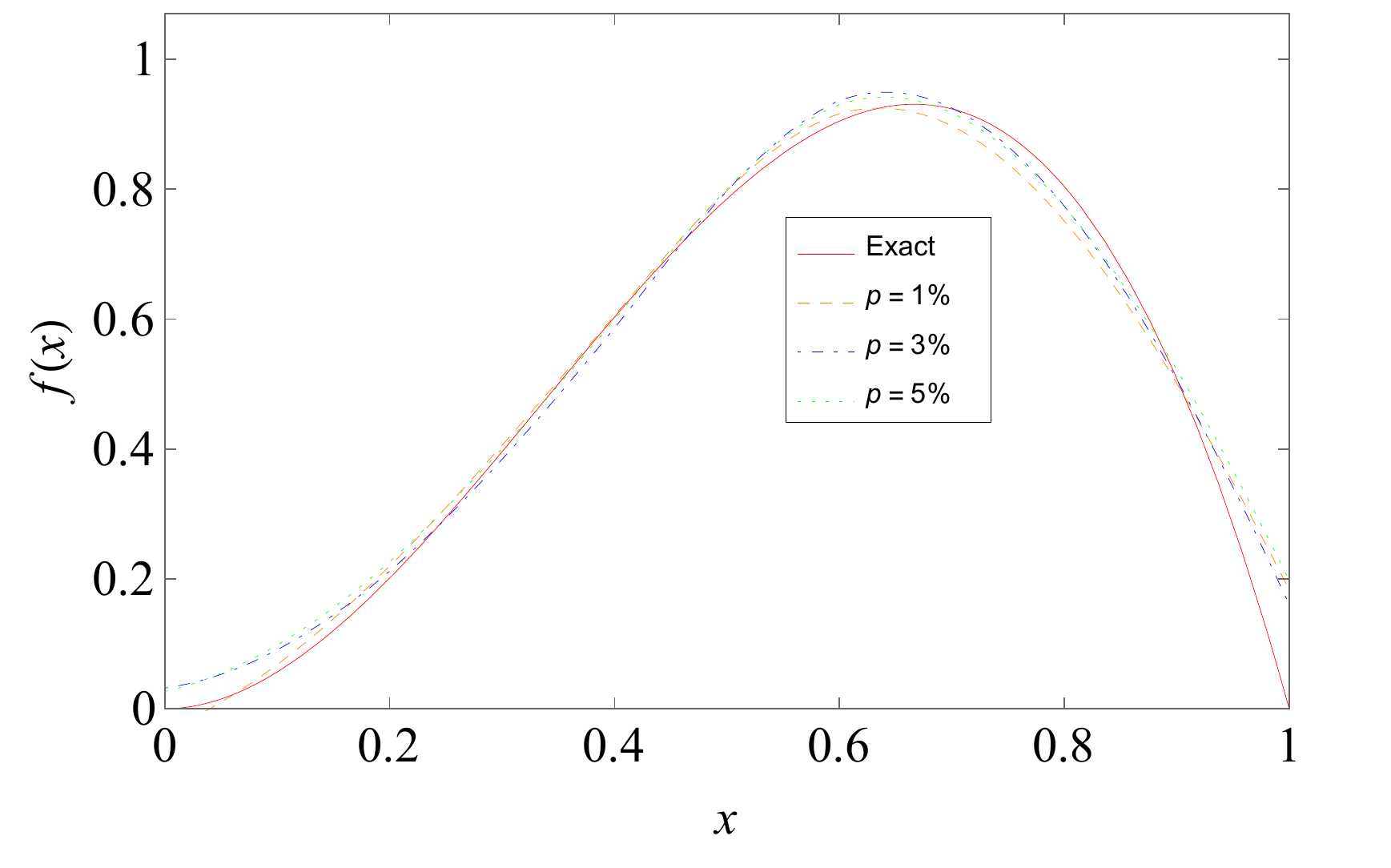}
\caption{Exact and recovered $f(x)$ by using Algorithm \ref{alg1} for $p\in \{1\%,3\%,5\%\}$, respectively.} \label{fig2}
\end{figure}

The algorithm stops at iterations $k\in \{11,19,17\}$, for $p\in \{1\%,3\%,5\%\}$, respectively.

\begin{table}[ht] 
\caption{Errors depending on the iteration number $k$ for noise free data ($p=0\%$).}
\centering
$\begin{array}{cccccc}
\hline
 k & 1 & 2 & 3 & 4 & 5 \\
\hline
 e\left(k,f_k\right) & 6.099\times 10^{-1} & 6.744\times 10^{-2} & 4.632\times 10^{-3} & 4.629\times 10^{-3} & 3.134\times 10^{-3} \\
\hline
 \text{E}\left(k,f_k\right) & 3.054\times 10^{-1} & 2.603\times 10^{-1} & 2.601\times 10^{-1} & 1.569\times 10^{-1} & 1.149\times 10^{-1} \\
\hline
\end{array}$ \label{tab2}
\end{table}

\subsection*{Example 3}
We take the exact source term as $$f(x)=\dfrac{1}{4}\left(\arctan\left(\dfrac{x}{\pi}\right)-\sin(2 \pi x)\right) +\dfrac{1}{2}, \; x\in (0,1).$$

\begin{figure}[ht] 
\centering
\includegraphics[scale=0.5]{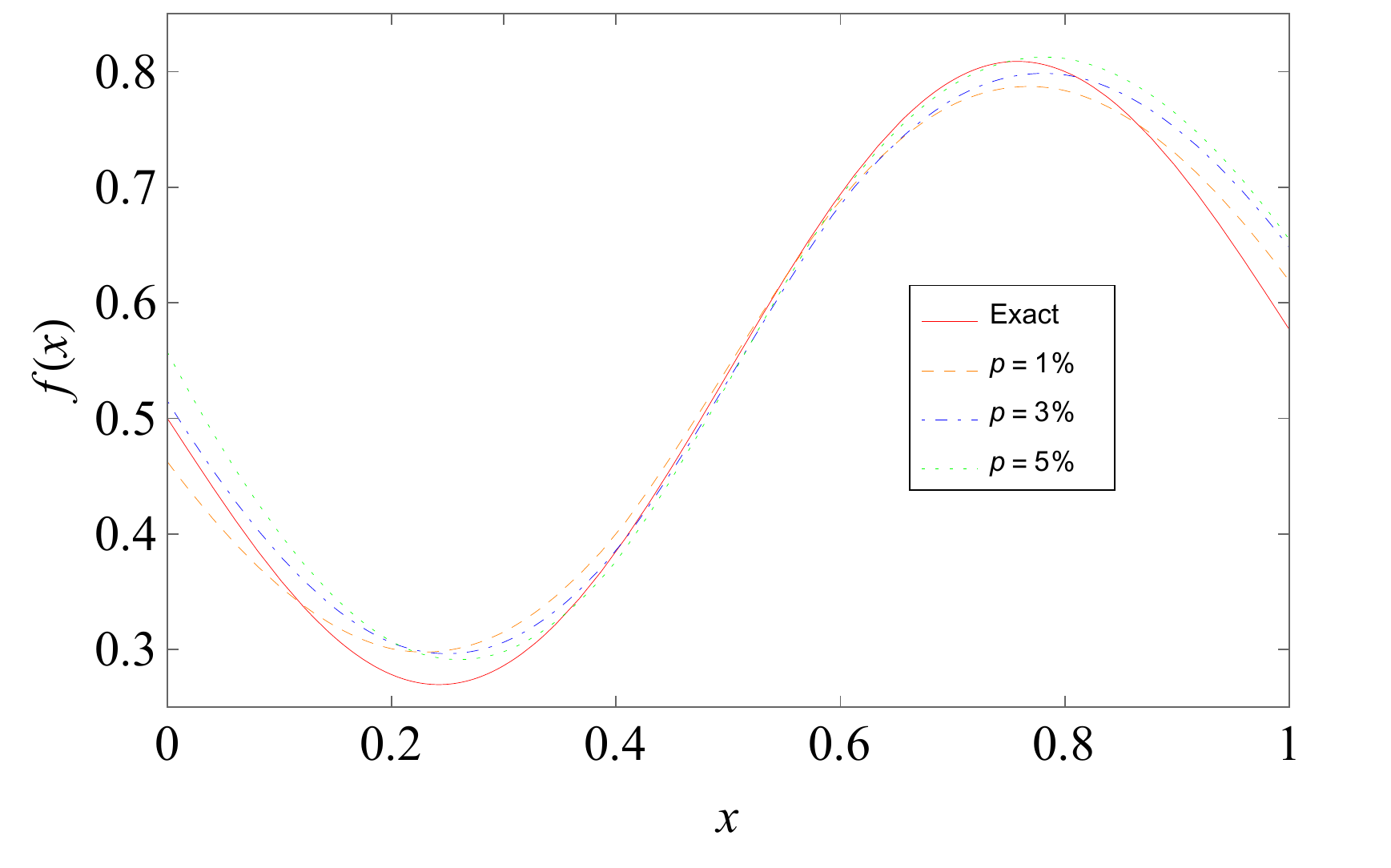}
\caption{Exact and recovered $f(x)$ by using Algorithm \ref{alg1} for $p\in \{1\%,3\%,5\%\}$, respectively.} \label{fig3}
\end{figure}
\newpage

The algorithm stops at iterations $k\in \{13,15,16\}$ for $p\in \{1\%,3\%,5\%\}$, respectively.

\begin{table}[ht] 
\caption{Errors depending on the iteration number $k$ for noise free data ($p=0\%$).}
\centering
$\begin{array}{cccccc}
\hline
 k & 1 & 2 & 3 & 4 & 5 \\
\hline
 e\left(k,f_k\right) & 6.265\times 10^{-1} & 5.936\times 10^{-2} & 2.904\times 10^{-4} & 2.858\times 10^{-4} & 2.858\times 10^{-4} \\
\hline
 \text{E}\left(k,f_k\right) & 1.77\times 10^{-1} & 1.077\times 10^{-1} & 1.077\times 10^{-1} & 1.077\times 10^{-1} & 1.076\times 10^{-1} \\
\hline
\end{array}$ \label{tab3}
\end{table}

From Figs. \ref{fig1}, \ref{fig2} and \ref{fig3}, we clearly see that the recovered source terms $f$ corresponding to different levels of random noise are not that far from the exact source terms. This effectively shows that the designed Algorithm \ref{alg1} yields numerically stable results. It also shows the regularizing effects of both the regularization parameter $\varepsilon$ and the CG method.

The error tables \ref{tab1}, \ref{tab2} and \ref{tab3} show that the convergence error and the accuracy error decrease as the iteration number $k$ increase.

\section{Conclusions and final remarks}\label{sec6}
We have studied an inverse source problem for identifying forcing terms from the terminal time data in a linear wave equation with dynamic boundary conditions. Using the weak solution approach, an optimization method has been adapted for the Tikhonov's cost functional. Then, an explicit gradient formula for the cost has been derived via the solution of an adequate adjoint system. The Lipschitz continuity of the gradient has been shown. Next, the existence and the uniqueness of a solution to the minimization problem have been discussed, and a sufficient condition for the uniqueness has been given. A numerical CG algorithm has been designed to recover an internal wave force.

\newpage
\begin{remark} We close the paper with the following remarks:
\begin{itemize}
    \item Although we have only considered a simplified model of hyperbolic systems with dynamic boundary conditions, our approach can be generalized to more general models as
    \begin{empheq}[left = \empheqlbrace]{alignat=2}
    \begin{aligned}
    &\partial_{t}^2 y -d \Delta y + a(x)y = F(t,x), &&\qquad \text{in } \Omega_T , \\
    &\partial_{t}^2 y_{\Gamma} -\gamma \Delta_{\Gamma} y_{\Gamma}+d\partial_{\nu} y + b(x)y_{\Gamma} = G(t,x), &&\qquad \text{on } \Gamma_T, \\
    &y_{\Gamma}(t,x) = y_{|\Gamma}(t,x), &&\qquad \text{on } \Gamma_T, \\
    &(y,y_{\Gamma})\rvert_{t=0}=(y_0, y_{0,\Gamma}),   &&\qquad \Omega\times\Gamma\\
    &(\partial_{t} y,\partial_{t} y_{\Gamma})\rvert_{t=0}=(y_1, y_{1,\Gamma}),   &&\qquad \Omega\times\Gamma,
    \end{aligned}
    \end{empheq}
    where $\Omega \subset \mathbb{R}^N$ $(N\le 3)$ is a bounded domain with smooth boundary $\Gamma$, $a\in L^\infty(\Omega)$, $b \in L^\infty(\Gamma)$, and $d, \gamma>0$ are given speed constants.
\item Comparing to the heat equation with dynamic boundary conditions, the proposed Landweber scheme in \cite{ACM'21'} becomes slow for the wave equation \eqref{eq1to4}. This issue can be interpreted in terms of the measured data we have considered. In our case, we have only used the final time data $Y(T, \cdot)$, while we might also add the final speed $Y_t(T, \cdot)$ as a measurement in view of \cite{Ha'09'}. This issue has been fixed by considering a different CG algorithm that yields fast and accurate numerical results.
\end{itemize}
\end{remark}





\end{document}